\documentclass[aop]{imsart}

\RequirePackage{amsthm,amsmath,amsfonts,amssymb}
\RequirePackage[numbers]{natbib}
\RequirePackage[colorlinks,citecolor=blue,urlcolor=blue]{hyperref}
\RequirePackage{graphicx}
\usepackage{bbm}
\startlocaldefs
\theoremstyle{plain}

\newtheorem{theorem}{Theorem}[section]
\newtheorem{lemma}[theorem]{Lemma}
\newtheorem{corollary}[theorem]{Corollary}
\newtheorem{conjecture}[theorem]{Conjecture}
\theoremstyle{remark}

\renewcommand{\endproof}{\hspace*{\fill}$\blacksquare$}
\newcommand{\pr}{{\mathbb{P}}}

\newcommand{\R}{{\mathbb{R}}}
\newcommand{\QQ}{{\mathcal{Q}}}

\newcommand{\N}{{\mathbb{N}}}

\newcommand{\Var}{\mathrm{Var}}
\def\={\stackrel {\rm def}  {=}}
\def\di={\overset{\text{${\mathcal{D} }$}} =}

\endlocaldefs

\begin{document}

\begin{frontmatter}
\title{The sharp form of the Kolmogorov--Rogozin inequality and a conjecture of Leader--Radcliffe}
\runtitle{}

\begin{aug}
\author[A]{\fnms{Tomas} \snm{Ju\v skevi\v cius}},
\address[A]{Institute of Mathematics of the Czech Academy of Sciences}

\end{aug}

\begin{abstract}
    Let $X$ be a random variable and define its concentration function by
$$\QQ_{h}(X)=\sup_{x\in \R}\pr(X\in (x,x+h]).$$
For a sum $S_n=X_1+\cdots+X_n$ of independent real-valued random variables the Kolmogorov-Rogozin inequality states that
$$\QQ_{h}(S_n)\leq C\left(\sum_{i=1}^{n}(1-\QQ_{h}(X_i))\right)^{-\frac{1}{2}}.$$
In this paper we give an optimal bound for $\QQ_{h}(S_n)$ in terms of $\QQ_{h}(X_i)$, which settles a question posed by Leader and Radcliffe in 1994. Moreover, we show that the extremal distributions are mixtures of two uniform distributions each lying on an arithmetic progression.     
\end{abstract}



\end{frontmatter}

\section{Introduction}

Define the concentration function $\QQ_{h}(X)$ of a random variable $X$ is defined to be the quantity
\begin{equation*}
\QQ_{h}(X)=\sup_{x\in \R}\pr(X\in (x,x+h]).
\end{equation*}
Let $S_{n}=X_{1}+\cdots+ X_{n}$ be a sum of independent random variables. Doeblin and Levy \cite{DoeblinLevy} were the first to study the spread of the distribution of $S_n$ in terms of its concentration function by establishing quantitative bounds on $\QQ_{h}(S_n)$ in terms of the individual concentration functions $\QQ_{h}(X_i)$ of the summands. Kolmogorov \cite{Kolmogorov} improved these results and obtained a bound that was asymptotically sharp up to a logarithmic factor in $n$. The latter result was improved by Rogozin \cite{KR} who removed the logarithmic factor. 
\begin{theorem} [Kolmogorov--Rogozin]\label{KOL}
There exists an absolute constant $C>0$ such that for $h>0$ we have
$$\QQ_{h}(S_n)\leq C\left(\sum_{i=1}^{n}(1-\QQ_{h}(X_i))\right)^{-\frac{1}{2}}.$$
\end{theorem}
The latter inequality is asymtotically sharp in the case $\QQ_{h}(X_i)=\alpha$ for fixed $\alpha$ as then the right hand side is $O(n^{-1/2})$. Yet for small values of $\alpha$ it degenerates. This deficiency was removed by Kesten \cite{Kesten} who inserted a certain multiplicative factor that provides the correct asymptotics of the bound as $\alpha \rightarrow 0$.\\

The goal of this paper is to establish the optimal upper bound in Theorem \ref{KOL}. Before stating our main result let us first adopt some conventions. For $k\in \N$ let us denote by $\nu^{k}$ the uniform distribution on the $k$ term arithmetic progression $\left\{-k+1,-k+3,\ldots,k-1\right\}$. For $\alpha\in[\frac{1}{k+1},\frac{1}{k}]$ we denote by $T(\alpha)$ a random variable having distribution 
$$(1-\tau)\nu^{k+1}+\tau \nu^{k}$$ 
with $\tau=k(k+1)\alpha-k$. In particular, the random variable $T(\frac{1}{k})$ has distribution $\nu^{k}$. Note that $\QQ_{2}(T(\alpha))$ is attained by any interval containing two consecutive atoms of $T(\alpha)$ and it easily follows that from the definition that $\QQ_{2}(T(\alpha))=\alpha$.\\ 

For any random variable $X$ and $a>0$ we have $\QQ_{h}(X)=\QQ_{h/a}(aX)$. We shall henceforth only treat the case $h=2$ for concentration functions of individual random variables under consideration and write $\QQ$ instead of $\QQ_{2}$. The reason for picking this particular value will soon become apparent. We now state the main result of the paper.

\begin{theorem}\label{main}
Let $X_1,\ldots, X_n,$ be independent random variables such that $\QQ_{2}(X_i)\leq\alpha_i\in[0,1]$. For all integer $\ell \geq 1$ we have
\begin{eqnarray*}\label{nel}
\QQ_{2\ell}\left(X_1+\cdots+X_n\right) &\leq& \QQ_{2\ell}\left(T_1(\alpha_1)+\cdots+T_n(\alpha_n)\right),\\
\end{eqnarray*}
where $T_{i}(\alpha_i)$ are independent.
\end{theorem}

The latter inequality is therefore optimal with the choice $X_i \buildrel d \over = T_{i}(\alpha_i)$ saturating the bound. It turns out that
$$\QQ_{2\ell}\left(T_1(\alpha_1)+\cdots+T_n(\alpha_n)\right)=\mathbb{P}\left(T_1(\alpha_1)+\cdots+T_n(\alpha_n) \in (-\ell,\ell]\right),$$
which is not obvious since $T_1(\alpha_1)+\cdots+T_n(\alpha_n)$ is not in general unimodal.

Leader and Radcliffe \cite{LR} established Theorem \ref{main} in the special case $\alpha_i=\frac{2}{k}$ for integer $k\geq 2$ and $l=1$ and posed the question to extend their inequality for other values of $\alpha$. Theorem \ref{main} thus resolves their question and also extends the desired result to all $\ell \in \N$ and possibly different $\alpha_i$'s. The choice $\alpha_i=\frac{1}{2}$ recovers the famous Littlewood--Offord inequality by Erd\H{o}s \cite{LO}. The case $\alpha_i=\frac{2}{k}$ is special as only in this case $T_i(\alpha)$ has a uniform distribution. The reason for picking $h=2$ is now clear --- it is the smallest natural number for which the maximizing distributions take integer values.\\

In most applications of the Kolmogorov--Rogozin type bounds are applied for the case $\alpha_i=\alpha$. For this case the Local Limit Theorem can be used to obtain a simple bound with the asymptotically sharp constant in Theorem \ref{main}.\\

\begin{corollary}\label{cor1}
Let $X_1,\ldots,X_n$ independent random variables with $\QQ(X_i)\leq \alpha\in(0,1)$. For $\ell \geq 1$ and $n\rightarrow \infty$ we have
\begin{equation*}
\QQ_{2\ell}(X_1+\cdots+X_n) \leq \frac{2\ell+o(1)}{\sqrt{2\pi\Var(T_{1}+\cdots+T_{n})}}.
\end{equation*}

\end{corollary}

We now turn to the situation where $\alpha_i\geq1/2$. Note that for $\alpha_i\geq 1/2$ the corresponding extremal random variables $T(\alpha)$ in Theorem \ref{main} have symmetric distribution on the set $\{-1,0,1\}$ so that $\pr(T_i(\alpha_i)=\pm 1)=1-\alpha_i$. We shall allow the values $\alpha_i$ depend on $n$ in the forthcoming result, which can be viewed as a more general version of Poisson type anticoncentration recently investigated by Fox, Kwan an Sauermann \cite{FKS} (see their Section 6) for linear combinations of Bernoulli random variables.

\begin{corollary}\label{kant}
For independent random variables $X_1,\ldots,X_n$ satisfying $\QQ(X_i)\leq\alpha_i\in [1/2,1]$ we have
$$\QQ(X_1+\ldots+X_n)\leq \pr(\eta_1-\eta_2 \in \{0,1\})\leq \sqrt{\frac{2}{\pi \lambda}},$$
where $\eta_i$ are independent copies of a Poisson random variable with mean $\lambda=\sum_{i=1}^{n}(1-\alpha_{i})$. 
\end{corollary}

Note that the first inequality is sharp in the following sense. If we do not restrict $n$ and fix the sum of $\alpha_i$'s, then this bound can be achieved in the limit as $n\rightarrow \infty$ by sums of distributions $T_i(1-\frac{\lambda}{n})$. The second inequality can be seen to be sharp as $\lambda\rightarrow \infty$ by the Local Limit Theorem.\\

\textbf{Remark.} We used the half-closed interval in the definition of $\QQ$, whereas some authors used both closed or open intervals of fixed length. It is straightforward to deduce the corresponding versions of Theorem \ref{main} with these different definitions as well. The reason for our choice is convenience --- with this definition the optimal interval of concentration is always symmetric with respect to the origin.\\

The paper is organized as follows. We start with giving an informal outline of the proof strategy of Theorem \ref{main} in Section $2$. We then proceed with establishing the necessary steps in the order described in the outline. Proofs of the corollaries are given and one open problem formulated in Section $6$.

\section{Outline of the proof of the main inequality} The proof proceeds in the following steps that are split into multiple small sections:\\
\begin{enumerate}
  \item The problem is reduced to \textbf{simple} random variables, i.e., random variables taking only finitely many values all with rational probabilities;\\
  \item It is shown that the distributions of simple random variables $X_i$ under the condition $\QQ(X_i)\leq \alpha_i$ can be expressed as a convex combination of 
uniform distributions with uniformities tied to the values $\alpha_i$ in the appropriate way (see Lemma \ref{repr});\\
  \item  Using a LYM-type inequality for extremal combinatorics we establish the desired inequality in the special case when $\alpha_i = 1/k_{i}$, $k_i\in \N$. \\
 \item  Finally, using the representation for measures provided Lemma \ref{repr} and multiple applications of the inequality established in the latter step we shall be done.
\end{enumerate}

Since the reduction of the first step is rather standard and dull, we postpone it to the Appendix. In the other parts we shall only deal with simple random variables. The remaining steps of the proof will appear in the order of the list above in the upcoming sections.

\section{The representation lemma}
For $k\geq 1$ define by $\mu^{k}$ a uniform distribution on some $k$ points in $\R$ that are pairwise at distance at least 2. Note that the definition of $\mu^{k}$ depends on the choice of those points, which is not reflected in the notation. Usually we will supply $\mu^{k}$ with a subscript, which will mean that the distributions with distinct subscripts might be concentrated in different sets. When the set of $k$ points will be $\left\{-k+1,-k+3,\ldots,k-3,k-1\right\}$, we are going to use the notation $\nu^{k}$ instead of $\mu^{k}$ in consistency with the definitions of the Introduction. Furthermore, for a random variable $X$ we shall denote it's probability distribution by $\mathcal{L}(X)$.\\

\begin{lemma}\label{repr}
Let $X$ be a simple random variable with $\QQ(X)=m/n\in(1/(k+1),1/k]$. Assume that $X$ is concentrated in the set $S=\left\{y_{1},\ldots,y_{M}\right\}$ with rational probabilities $\mathbb{P}\left(X=y_{i}\right)=m_{i}/n_{i}$. Let us define
 $$N=n\prod_{i}n_{i},\quad K=(n-km)\prod_{i}n_{i},\quad L= ((k+1)m-n)\prod_{i}n_{i}.$$ 
We can express the distribution of $X$ as
\begin{equation*}
\mathcal{L}(X)=\frac{1-\tau}{K}\sum^{K}_{l=1}\mu_{l}^{k+1}+\frac{\tau}{L}\sum^{K+L}_{l=K+1}\mu_{l}^{k},
\end{equation*}
where $\tau=k(k+1)m/n-k$.
\end{lemma}

\begin{proof}

Assume that $y_1\leq \ldots \leq y_{M}$. We can regard the distribution of $X$ as the uniform distribution on a multiset $S'$, where $S'$ is obtained from $S$ by taking the element $y_{i}$ exactly $nm_{i}\prod_{j\neq i}n_{i}$ times. Let ${x_{1},\ldots,x_{N}}$ be the elements of $S'$ in increasing order. \\

The condition $\QQ(X)=m/n$ ensures than no more than $d=Nm/n$ points lie in the interval $(x,x+2]$ for all $x$. Thus the points $x_{l},x_{l+d}$ are at distance at least $2$. For $l\leq L$ the points $x_{l},x_{l+d},\ldots,x_{l+kd}$ are pairwise at distance at least two. Each point has mass $1/N$, so in order make the measure on the latter set of points into a probability measure we must divide it by it by $(k+1)/N$. We have 
$$(k+1)/N=(k+1)(n-km)/(nK)=(1-(k(k+1)m/n-k))/K=(1-\tau)/K,$$
thus obtaining the first $K$ distributions $\mu_{l}^{k+1}$ with the desired weights. \\

For $K+1\leq l\leq K+L$ take the points $x_{l},x_{l+d},\ldots,x_{l+(k-1)d}$ and the measures concentrated on those points will give us the required $L$ measures $\mu_{l}^{k}$. It can be checked that the proportion is again correct, but that will follow from the fact that we used up all points from $S'$ and took each of them only once. Indeed, we started constructing each measure in the representation from a different point in ${x_{1},\ldots,x_{K+L}}$ and then added points with equally spaced indices. Thus we did not use any point twice. Furthermore, $K(k+1)+Lk=N$ and so we used them all.

\end{proof}

\section{The case $\alpha_i=\frac{1}{k_i}$ and a LYM type inequality for multisets}

In this section we shall be dealing with multisets defined on the ground set $[n]$ such that each element has an upper bound, say $k_i$, on its multiplicity. The case $k_i=1$ naturally reduces to the study of sets. In the latter case we can switch between talking about the powerset of $[n]$ to the study of indicator vectors in $\left\{0,1\right\}^n$ with set inclusion corresponding to the product order in $\left\{0,1\right\}^n$. \\

Analogously, we shall view multisets as vectors in the discrete rectangle $L(k_1,\ldots,k_n)=\left\{0,\ldots,k_1-1\right\}\times \cdots \times \left\{0,\ldots,k_n-1\right\}$ by associating with a multiset the vector of multiplicities of each element in it. \\

For a vector $x\in \R ^{n}$ we shall denote its $i$-th coordinate by $x_i$. We shall endow $L(k_1,\ldots,k_n)$ with the product order. That is, $v\leq w$ if and only if $v_i\leq w_i$. Multiset inclusion corresponds to this order as in the case with sets.\\

We shall call a collection of vectors $v_1,\ldots,v_k$ a \textbf{chain} if $v_1\leq \cdots \leq v_k$  and refer to the number $k$ as its length. We say that a family of vectors $\mathcal{F}$ is \textbf{$k$-Sperner} if it has no chains of length $k+1$. In the case $k=1$ we shall say that $\mathcal{F}$ is an \textbf{antichain} rather than $1$-Sperner.\\

Let us partition $L(k_1,\ldots,k_n)$ into classes $L_i$ where 
$$L_i=\left\{x\in L(k_{1},\ldots,k_{n})\, | \, x_1+\cdots+ x_n =i\right\}.$$
 Note that $|L_i|$ is a symmetric sequence in the sense that $|L_i|=|L_{N-i}|$ where $N=\sum (k_i-1)$. The sequence $|L_i|$ is non-decreasing for $i\leq \left\lfloor \frac{N}{2}\right\rfloor$ and thus, by symmetry, it is non-increasing for $i\geq \left\lceil  \frac{N}{2}\right\rceil$.\\

For $k\leq k_1+\cdots+k_n+1$ write $f(k_1,k_2,\ldots,k_n,k)$ for the sum of the $k$ largest sets $L_{i}$. These are just the $k$ middle diagonals of the rectangle $L(k_1,\ldots,k_n)$. \\

Similar to Erd\H{o}s's proof of the Littlewood-Offord problem in \cite{ELO} that used a Sperner type theorem we shall use a similar result for multiset $k$-Sperner families in the exactly the same way. The result we shall need is the following.

\begin{lemma}\label{BKT}
Let $\mathcal{F}$ be a $k$-Sperner family of vectors in $L(k_1,\ldots,k_n)$. Then
\begin{equation*}
\left|\mathcal{F}\right| \leq f(m_1,m_2,\ldots, m_n,k).
\end{equation*}
\end{lemma}

Before we proceed with the proof, let us state a standard LYM type inequality for antichains of multisets.

\begin{theorem} \label{LYM}
Let $\mathcal{F}$ be an antichain in $L(k_{1},\ldots,k_{n})$. For $0\leq i \leq \sum_{j=1}^{n}(k_{j}-1)$ denote $\mathcal{F}_{i}=\mathcal{F} \cap L_{i}$. We have
\begin{equation*}
\sum_{i}^{} \frac{|\mathcal{F}_{i}|}{|L_{i}|}\leq 1.
\end{equation*} 
\end{theorem}

The proof of Theorem \ref{LYM} can be found in Chapter $10$ of the book by Anderson \cite{Anderson}. In the case of sets it is known as the LYM inequality (Lubell-Yamamoto-Meshalkin).\\

{\em Proof of Lemma \ref{BKT}.} Let $\mathcal{F}$ be a $k$-Sperner family. It is easy to see that $\mathcal{F}$ is a union of $k$ antichains. Indeed, the maximal elements of $\mathcal{F}$ form an antichain and the remaining elements form a $(k-1)$-Sperner family and so the observation follows by induction on $k$. Let $\mathcal{A}$ be one of the $k$ antichains that decompose $\mathcal{F}$.\\

Using Theorem \ref{LYM} we obtain
\begin{equation*}
\sum_{i}^{} \frac{|\mathcal{A}_{i}|}{|L_{i}|}\leq 1.
\end{equation*}

Summing this inequality over all $k$ antichains we obtain
\begin{equation}\label{sumaaa}
\sum_{i}^{} \frac{|\mathcal{F}_{i}|}{|L_{i}|}\leq k.
\end{equation}

 For families of vectors of fixed cardinality the sum in \eqref{sumaaa} is minimized by families containing vectors with coordinate sums as close to $\sum_{i}(k_i -1)/2$ as possible. This is because in view of \eqref{sumaaa} the vectors are assigned the smallest weight.\\

Suppose now that $|\mathcal{F}|> f(k_1,\ldots,k_n,k)$. Note for the family of vectors consisting of the middle $k$ diagonals of $L(k_{1},\ldots,k_{n})$ the corresponding sum in \eqref{sumaaa} is exactly equal to $1$ and is minimal among all families having $f(k_1,\ldots,k_n,k)$ vectors. Therefore for any family of vectors with more elements the corresponding sum in \eqref{sumaaa} is strictly greater than $1$, which is a contradiction. Thus $|\mathcal{F}|\leq f(k_1,\ldots,k_n,k)$ and we are done. \endproof

We can now obtain the statement of Theorem \ref{main} in an important special case, which is a slight generalization of the main result of Leader and Radcliffe \cite{LR}.
\begin{lemma}\label{LRlemma}
Let $X_1,\ldots,X_n$ be independent simple random variables such that $\QQ(X_i)=1/k_i$. For all integer $\ell\geq 1$ and $x\in \R$ we have 
\begin{equation*}
\mathbb{P} \left(X_1+\cdots+X_n \in (x-\ell,x+\ell]\right) \leq \mathbb{P}\left(T_1(1/k_1)+\cdots+T_n(1/k_n) \in (-\ell,\ell]\right),
\end{equation*}
where $T_{i}(1/k_i)$ are independent.

\end{lemma}

{\em Proof of Lemma \ref{LRlemma}.} In view of Lemma \ref{repr} we can assume that $\mathcal{L}(X_i)=\mu_{i}^{k_i}$. This is due to the fact that our optimization problem is linear with respect to the measures in the decomposition given by Lemma \ref{repr} - we can therefore pick the measure of type $\mu^{k_i}$ in the decomposition that maximizes the functional in question. For each $i$ let us denote the values $X_i$ takes by $x_{i,1},\ldots,x_{i,k_i}$. Let us define a family of vectors (or multisets) 
\begin{equation*}
\mathcal{F}=\left\{v\in L(k_1,\ldots,k_n)| \sum_{j=1}^{n}x_{j,v_{j}}\in (x-\ell,x+\ell]\right\}. 
\end{equation*}
Note that by definition of measures $\mu_{i}^{k_i}$ the points $x_{i,1},\ldots,x_{i,k_i}$ are all at distance at least $2$ within each other. Therefore if we had a chain of vectors (or multisets) of length $\ell+1$ then the sums corresponding to the top and bottom vectors (or multisets) would differ by strictly more than $2\ell$ and so we get a contradiction. Therefore the family $\mathcal{F}$ is $\ell$-Sperner.\\

Using Lemma \ref{BKT} we therefore have
\begin{eqnarray*}
\mathbb{P} \left(X_1+\cdots+X_n \in (x-\ell,x+\ell]\right)&=&|\mathcal{F}|/\prod_{j=1}^{n}k_i\\
&\leq& f(k_1,k_2,\ldots, k_n,\ell)/\prod_{j=1}^{n}k_i\\
&=&\mathbb{P}\left(T_1(1/k_1)+\cdots+T_n(1/k_n) \in (-\ell,\ell]\right).
\end{eqnarray*}

\section{Putting it all together}

Let $X_{1},\ldots,X_{n}$ be independent random variables such that $\QQ(X_i)\leq \alpha_i$. Since the bound of Theorem \ref{main} is continuous with respect to the variables $\alpha_i$, we can consider only rational $\alpha_i =m_{i}/n_{i} \in (1/(k_{i}+1),1/k_{i}]$. We can without loss of generality assume that the random variables $X_i$ are simple (see the Appendix). Thus each random variable $X_i$ takes values in some finite set $\left\{y_{i,1},\ldots, y_{i,D}\right\}$ (we can take one value of $D$ for all variables by adding some points with $0$ probability). 

Moreover, we can assume that the probabilities $\mathbb{P}\left(X_{i}=y_{i,k}\right)$ are also rational. Thus $\mathbb{P}\left(X_{i}=y_{i,k}\right)=m_{i,k}/n_{i,k}$. Writing $N_{i}=n_{i}\prod_{j=1}^{n}n_{i,j}$ we can look at the distribution of $X_{i}$ as a uniform distribution on a multiset with $N_{i}$ elements. By Lemma \ref{repr} we have\\
\begin{equation*}
\mathcal{L}(X_{i})=\frac{1-\tau_{i}}{K_{i}}\sum_{l_{i}=1}^{K_{i}}\mu^{k_{i}+1}_{i,l_{i}}+\frac{\tau_{i}}{L_{i}}\sum_{l_{i}=K_{i}+1}^{K_{i}+L_{i}}\mu^{k_{i}}_{i,l_{i}},
\end{equation*}  
where $K_i, L_i$ and $\tau_{i}$ are defined as in Lemma \ref{repr}. \\

 We shall expand the product measure $\prod_{i=1}^{n}\mathcal{L}(X_{i})$ into a sum of products of the measures $\mu_{i,l_{i}}^{\tilde{k}_{i}}$, where $\tilde{k}_{i}=k_{i}+1$ for $l_{i} \leq K_{i}$ and $\tilde{k}_{i}=k_{i}$ otherwise. For the same ranges of $l_{i}$ define $\tilde{\tau_{i}}$ in a natural way - the coefficient in front of $\mu_{i,l_{i}}^{\tilde{k_{i}}}$. Expanding the product measure and using Lemma \ref{LRlemma} term by term we obtain

\begin{eqnarray*}
&&\mathbb{P}\left(X_{1}+\cdots+X_{n}\in (x-\ell,x+\ell]\right)\\
&=&\prod_{i=1}^{n}\mathcal{L}(X_{i})((x-\ell,x+\ell])\\
&=&\prod_{i=1}^{n} \left(\frac{1-\tau_{i}}{K_{i}}\sum_{l_{i}=1}^{K_{i}}\mu^{k_{i}}_{i,l_{i}}+\frac{\tau_{i}}{L_{i}}\sum_{l=K_{i}+1}^{K_{i}+L_{i}}\mu^{k_{i}+1}_{i,l_{i}}\right)((x-\ell,x+\ell])\\
&=&\prod_{i=1}^{n}\left(\tilde{\tau_{i}}\sum_{l=1}^{K_{i}+L_{i}}\mu^{\tilde{k}_{i}}_{i,l}\right)((x-\ell,x+\ell])\\
&=&\sum_{l_{1},\ldots,l_{n}}\prod_{i=1}^{n}\tilde{\tau_{i}}\mu^{\tilde{k}_{i}}_{i,l}((x-\ell,x+\ell])\\
&\leq& \sum_{l_{1},\ldots,l_{n}}\prod_{i=1}^{n}\tilde{\tau_{i}}\nu^{\tilde{k}_{i}}_{i,l}((-k,k])\\
&=&\prod_{i=1}^{n}\left(\tilde{\tau_{i}}\sum_{l=1}^{K_{i}+L_{i}}\nu^{\tilde{k}_{i}}_{i,l}\right)((-k,k])\\
 &=&\prod_{i=1}^{n}\left(\frac{1-\tau_{i}}{K_{i}}\sum_{l_{i}=1}^{K_{i}}\nu^{k_{i}}_{i,l_{i}}+\frac{\tau_{i}}{L_{i}}\sum_{l=K_{i}+1}^{K_{i}+L_{i}}\nu^{k_{i}+1}_{i,l_{i}}\right)((-k,k])\\
&=&\prod_{i=1}^{n}(\tau_{i}\nu_{i}^{k_{i}}+(1-\tau_{i})\nu_{i}^{k_{i}+1})((-k,k])\\
&=&\mathbb{P}\left(T_{1}(\alpha_1)+\cdots+T_{n}(\alpha_n)\in (-k,k]\right).
\end{eqnarray*}

This completes the proof of Theorem \ref{main} and even provides the interval has the most mass under the distribution of $T_{1}(\alpha_1)+\cdots+T_{n}(\alpha_n)$.

\section{Proofs of the corollaries}

Corollary \ref{cor1} follows from Theorem \ref{main} by the Local Limit Theorem when $\frac{1}{\alpha}$ is not an integer. When $\frac{1}{\alpha}\in \N$ it follows either by continuity of the bound or one can use the Local Limit Theorem for the rescaled random variables $T_i(\alpha)$ after the application of Theorem \ref{main}.\\

{\em Proof of Corollary \ref{kant}.} When $\alpha_i\in [1/2,1]$ the sharp inequality of Theorem \ref{main} is given in terms of symmetric distributions. For symmetric independent real valued random variables $X_i$ such that $\mathbb{P}(|X_i|\geq 1)=2(1-\alpha_i)$ Kanter's inequality (see Corollary $1.3$ in \cite{Mattner}) states that for all $x\in \R$ we have
$$\mathbb{P}(|X_1+\cdots +X_n-x|<1)\leq \pr(\eta_1-\eta_2 \in \{0,1\}),$$
where $\eta_i$ are independent copies of a Poisson random variable with mean $\lambda=\sum_{i=1}^{n}(1-\alpha_{i})$.\\

To finish up, just notice that for the extremal distribution $T_{1}(\alpha_1)+\cdots+T_{n}(\alpha_n)$, coming from Theorem \ref{main}, has the largest probability in the interval of the form $(x,x+2]$ exactly as the maximal probability of the same length open interval. This means we can apply Kanter's inequality to obtain the desired result.\\

The second inequality was established by Mattner and Roos (Lemma $1.4$ in \cite{Mattner}).\\

Let us formulate a natural conjecture regarding an analytically simpler form of the bound in Theorem \ref{main} that is of the flavor of the Corollary \ref{kant}. Kanter \cite{kanter} actually proved a more detailed result - he showed, in our notation, that for $\alpha_i\geq 1/2$ the value of $\QQ(T_{1}(\alpha_1)+\cdots+T_{n}(\alpha_n))$ is at most $\QQ(T_{1}(\alpha)+\cdots+T_{n}(\alpha))$, where $\alpha=\frac{1}{n}\sum_{i}\alpha_i$. If $\alpha_{i}\geq 3/4$ the result follows by majorization techniques or by using the positivity of the corresponding characteristic functions. The general case is much more involved and follows by quite delicate analysis of certain integrals (see \cite{Mattner}). It is tempting to believe that a similar phenomenon remains true in the unrestricted case. Let us be more precise. In the case $\alpha_i\geq 1/2$ the variances of $T_i(\alpha_i)$ are linear in $\alpha_i$ and so averaging the values $\alpha_i$ is the same as averaging the variances of the random variables. We thus have the following conjecture.

\begin{conjecture}
For any $\alpha_i\in [0,1]$ we have
$$\QQ(T_{1}(\alpha_1)+\cdots+T_{n}(\alpha_n))\leq \QQ(T_{1}(\alpha)+\cdots+T_{n}(\alpha)),$$
where all random variables are independent and $\alpha$ is the unique value in $[0,1]$ such that
$$\Var(T_{1}(\alpha_1)+\cdots+T_{n}(\alpha_n))=\Var(T_{1}(\alpha)+\cdots+T_{n}(\alpha)).$$
\end{conjecture}

\begin{funding}
The author was supported by the Czech Science Foundation, grant number 18-01472Y.
\end{funding}

\bibliographystyle{imsart-nameyear.bst}
\bibliography{mybibliography}

\section{Appendix}

The goal of this section is to show that the all random variables $X_i$ in Theorem \ref{main} can be assumed to be simple. To achieve this reduction we will approximate the distribution functions under consideration uniformly by distribution functions of simple random variables. Note that weak convergence is insufficient for our purposes as weak convergence does not imply the convergence of concentration functions. The required approximation comes from a well known fact in real analysis - we can approximate any bounded measurable functions by step functions, giving us the required discretization. Let us be more precise. Consider a random variable $X$ with distribution function $F(t)=\mathbb{P}\left(X\leq t\right)$. For all $m\in \N$ and $k=0,1,\ldots,m$ define the level sets 
$$A_k=\left\{t:\,\, F(t) \in \left(\frac{k-1}{m},\frac{k}{m}\right]\right\}.$$ 
The sets $A_k$ are intervals (possibly infinite) as $F$ is monotone. Furthermore, we define the sequence of functions $F_m$ by setting
\begin{equation*}
F_{m}(t)=\sum_{k=0}^{m}\frac{k}{m}\mathbbm{1}_{A_k}.
\end{equation*}

Each function $F_m$ is monotone and is a distribution function since

\begin{equation*}
\lim_{t\rightarrow \infty}F_{m}(t)=1 \quad \text{and} \quad \lim_{t\rightarrow -\infty}F_{m}(t)=0.
\end{equation*}

Consider the corresponding sequence of random variables $X^{(m)}$ with distribution function $F_{m}$. Since $F_m$ is a step function with differences between consecutive steps $\frac{1}{m}$ it follows that $X^{(m)}$ have a uniform distribution on a finite set. Furthermore, by the definition of the sequence $F_m$ we have that for all $t\in \R$

\begin{equation*}
\left|F(t)-F_{m}(t)\right|\leq \frac{1}{m}.
\end{equation*}

 It follows that $X^{(m)}$ converge to $X$ as $m\rightarrow \infty$ uniformly. It immediately follows that
\begin{eqnarray*}
\left|\QQ_{h}(X)-\QQ_{h}(X^{(m)})\right|&\leq& \sup_{t}\left|\left(F(t+h)-F(t)\right)-\left(F_{n}(t+h)-F_{n}(t)\right)\right|\\
&\leq&\sup_{t}\left|F(t+h)-F(t)\right|+\sup_{t}\left|F(t+h)-F(t)\right|\\
&\leq& \frac{2}{m}.
\end{eqnarray*}

We will be essentially done if we establish an analogous uniform approximation statement for sums of random variables. Fix $\varepsilon>0$ and for a sum of independent random variables $S_n=X_1+\cdots+X_n$ associate a corresponding discretized sum $S^{(m)}_n=X^{(m)}_{1}+\cdots+X^{(m)}_{n}$ as described above. Pick $m$ so that $2n/m<\varepsilon$. We have established the fact that on any interval $I=(x,x+h]$ we have $|\mathbb{P}(X^{(m)}_i \in I) - \mathbb{P}(X_{i}\in I)|\leq \frac{2}{m}$. \\

For $i=0,\ldots, n-1$ introduce an auxiliary sums of random variables
$$M^{i}_n=X^{(m)}_1+\cdots+X^{(m)}_i+X_{i+1}+\cdots+X_n.$$
We have that 
$$\mathbb{P}(S_n\in I)-\mathbb{P}(S^{(m)}_n\in I)= \sum_i \left(\mathbb{P}(M^{i}_n\in I)-\mathbb{P}(M^{i+1}_n\in I)\right).$$
By the triangle inequality we thus have 
$$\left|\mathbb{P}(X_1+\cdots+X_n\in I)-\mathbb{P}(X^{(m)}_1+\cdots+X^{(m)}_n\in I)\right|\leq \sum_{i}\left|\mathbb{P}(M^{i}_n\in I)-\mathbb{P}(M^{i+1}_n\in I)\right|.$$
Notice that $M^{i}_{n}$ and $M^{i+1}_{n}$ differ only in the $(i+1)-$th variable. Thus, conditioning on the outcomes of the remaining $n-1$ variables we have 
that each difference of probabilities in the latter sum are at most $\frac{2}{m}$ from the approximation result with a single random variable. Since the bound is true for conditional probabilities it also holds for unconditional ones. We therefore obtain
$$|\mathbb{P}(X_1+\cdots+X_n \in I)-\mathbb{P}(X^{(m)}_1+\cdots+X^{(m)}_n \in I)| \leq \frac{2n}{m} <\varepsilon.$$

We have to also discuss one last detail. After the discretization of a random variable $X$ we may slightly alter $\QQ(X)$. This effect turns out to be negligible in the context of Theorem \ref{main}. Indeed, notice that the upper bound in the theorem is continuous with respect to the values $\alpha_i$. This can be easily seen by taking the expectation with respect to $T_i$ - it then becomes a linear function of $\alpha_i$. By continuity we can thus assume that the discretized versions $X^{(m)}_{i}$ satisfy $\QQ(X^{(m)}_{i})\leq \alpha_i$ instead of a bound $\alpha_i+\frac{2}{m}$ which we get after discretizing which allows us to neglect some clutter in the proofs.

\end{document}